\documentclass[a4paper, 10pt]{article}

\usepackage[latin1]{inputenc}
\usepackage{subfigure}
\usepackage{amsmath}
\usepackage{amssymb}
\usepackage{amsthm}
\usepackage{amsfonts}
\usepackage[cyr]{aeguill}
\usepackage{xspace}
\usepackage[english]{babel}
\usepackage{graphicx}
\usepackage{float}
\usepackage{mathtools}
\usepackage{psfrag,floatflt,hyperref}
\usepackage{verbatim}
\usepackage{geometry}
\usepackage{multirow}
\usepackage[retainorgcmds]{IEEEtrantools}
\usepackage{stmaryrd}
\SetSymbolFont{stmry}{bold}{U}{stmry}{m}{n}
\usepackage{arydshln,leftidx,mathtools}
\usepackage{color}

\newtheorem*{theorem*}{Theorem}

\newtheorem*{proposition*}{Proposition}

\newtheorem*{corollary*}{Corollary}

\title{A formula for the energy of circulant graphs with two generators\footnote{The author acknowledges support from the Swiss NSF grant $200021\_132528/1$.}}
\author{Justine Louis}
\date{$10$th August $2015$}

\begin{document}
        \maketitle
\begin{abstract}
In this note, we derive closed formulas for the energy of circulant graphs generated by $1$ and $\gamma$, where $\gamma\geqslant2$ is an integer. We also find a formula for the energy of the complete graph without a Hamilton cycle.
\end{abstract}
\begin{center}
\line(1,0){250}
\end{center}
\par\vspace{\baselineskip}
Let $1\leqslant\gamma_1\leqslant\cdots\leqslant\gamma_d$ be integers. The circulant graph $C^{\gamma_1,\ldots,\gamma_d}_n$ generated by $\gamma_1,\ldots,\gamma_d$ on $n$ vertices labelled $0,1,\ldots,n-1$, is the $2d$-regular graph such that for all $v\in\mathbb{Z}/n\mathbb{Z}$, $v$ is connected to $v+\gamma_i$ $\textnormal{mod }n$ and to $v-\gamma_i$ $\textnormal{mod }n$, for all $i=1,\ldots,d$. The adjacency matrix $A=(A_{ij})$ of a graph on $n$ vertices is the $n\times n$ matrix with rows and columns indexed by the vertices such that $A_{ij}$ is the number of edges connecting vertices $i$ and $j$. Let $\lambda_k$, $k=1,\ldots,n$, denote the eigenvalues of the adjacency matrix. The energy of a graph $G$ on $n$ vertices is defined by the sum of the absolute value of the eigenvalues of $A$, that is
\begin{equation*}
E(G)=\sum_{k=1}^n\lvert\lambda_k\rvert.
\end{equation*}
The energy of circulant graphs and integral circulant graphs is widely studied, see for example \cite{MR2401632,MR2851448,MR2210213,stevanovic2005remarks}. It has interesting applications in theoretical chemistry, namely, it is related to the $\pi$-electron energy of a conjugated carbon molecule, see \cite{MR2977757}. In the following theorem, we give a formula for the energy of circulant graphs with two generators, $1$ and $\gamma$, $\gamma\geqslant2$. The formula is interesting as $n$ is larger than $\gamma$. 
\begin{theorem*}
Let $D_n(x)$ denote the Dirichlet kernel. The energy of the circulant graph $C^{1,2}_n$ is given by
\begin{equation*}
E(C^{1,2}_n)=4\left(D_{\lfloor n/6\rfloor}(2\pi/n)+D_{\lfloor n/6\rfloor}(4\pi/n)\right).
\end{equation*}
For $\gamma\geqslant3$, the energy of the circulant graph $C^{1,\gamma}_n$ is given by
\begin{equation*}
E(C^{1,\gamma}_n)=4\sum_{m\in\{1,\gamma\}}\left(\sum_{l=0}^{\lceil\gamma/2\rceil-1}D_{\lfloor(2l+1)n/(2(\gamma+1))\rfloor}\left(\frac{2\pi m}{n}\right)-\sum_{l=0}^{\lceil\gamma/2\rceil-2}D_{\lfloor(2l+1)n/(2(\gamma-1))\rfloor}\left(\frac{2\pi m}{n}\right)\right)
\end{equation*}
where $\lfloor x\rfloor$ denotes the greatest integer smaller or equal to $x$ and $\lceil x\rceil$ denotes the smallest integer greater or equal to $x$.
\end{theorem*}
\begin{proof}
The adjacency matrix of a circulant graph is circulant, it follows that the eigenvalues of $C^{1,\gamma}_n$ are given by $\lambda_k=2\cos(2\pi k/n)+2\cos(2\pi\gamma k/n)$, $k=0,\ldots,n-1$, (see \cite{MR1271140}). The energy of $C^{1,\gamma}_n$ is then given by
\begin{equation*}
E(C^{1,\gamma}_n)=2\sum_{k=0}^{n-1}\lvert\cos(2\pi k/n)+\cos(2\pi\gamma k/n)\rvert.
\end{equation*}
Let $\gamma=2$. The two roots of the equation $\cos{x}+\cos(2x)=0$ for $x\in\left[0,\pi\right]$ are $\pi/3$ and $\pi$. We write the energy as
\begin{align*}
E(C^{1,2}_n)&=4+4\sum_{k=1}^{\lceil n/2\rceil-1}\lvert\cos(2\pi k/n)+\cos(4\pi k/n)\rvert\\
&=4+4\sum_{k=1}^{\lfloor n/6\rfloor}(\cos(2\pi k/n)+\cos(4\pi k/n))-4\sum_{k=\lfloor n/6\rfloor+1}^{\lceil n/2\rceil-1}(\cos(2\pi k/n)+\cos(4\pi k/n)).
\end{align*}
The sum of $\cos(kx)$ over consecutive $k$'s can be expressed in terms of the Dirichlet kernel, namely
\begin{equation*}
D_n(x)=1+2\sum_{k=1}^n\cos(kx)=\frac{\sin((n+1/2)x)}{\sin(x/2)}.
\end{equation*}
As a consequence,
\begin{equation*}
2\sum_{k=n+1}^m\cos(kx)=D_m(x)-D_n(x).
\end{equation*}
The energy of $C^{1,2}_n$ is thus given by
\begin{equation*}
E(C^{1,2}_n)=4D_{\lfloor n/6\rfloor}(2\pi/n)+4D_{\lfloor n/6\rfloor}(4\pi/n)-2D_{\lceil n/2\rceil-1}(2\pi/n)-2D_{\lceil n/2\rceil-1}(4\pi/n).
\end{equation*}
The formula then follows from the fact that for odd $n$, $D_{(n-1)/2}(2\pi m/n)=0$ for $m=1,2$, and for even $n$, $D_{n/2-1}(2\pi/n)=1$ and $D_{n/2-1}(4\pi/n)=-1$.\\
Let $\gamma\geqslant3$. For odd $\gamma$, the $\gamma$ solutions of the equation $\cos{x}+\cos{\gamma x}=0$ for $x\in\left[0,\pi\right]$ are given in the increasing order by $\pi/(\gamma+1),\pi/(\gamma-1),3\pi/(\gamma+1),3\pi/(\gamma-1),\ldots,(\gamma-2)\pi/(\gamma-1),\gamma\pi/(\gamma+1).$ For even $\gamma$, they are given by $\pi/(\gamma+1),\pi/(\gamma-1),3\pi/(\gamma+1),3\pi/(\gamma-1),\ldots,(\gamma-3)\pi/(\gamma-1),(\gamma-1)\pi/(\gamma+1),\pi$. Let $n$ be odd. We split the sum over $k$ of cosines to group the positive terms together and the negative terms together. The energy is given by
\begin{align}
E(C^{1,\gamma}_n)&=4+4\sum_{k=1}^{(n-1)/2}\lvert\cos(2\pi k/n)+\cos(2\pi\gamma k/n)\rvert\nonumber\\
&=4+4\sum_{k=1}^{\lfloor n/(2(\gamma+1))\rfloor}(\cos(2\pi k/n)+\cos(2\pi\gamma k/n))\nonumber\\
&\ \ \ +4\sum_{l=0}^{\lceil\gamma/2\rceil-2}\sum_{k=\lfloor(2l+1)n/(2(\gamma-1))\rfloor+1}^{\lfloor(2l+3)n/(2(\gamma+1))\rfloor}(\cos(2\pi k/n)+\cos(2\pi\gamma k/n))\nonumber\\
&\ \ \ -4\sum_{l=0}^{\lceil\gamma/2\rceil-1}\sum_{k=\lfloor(2l+1)n/(2(\gamma+1))\rfloor+1}^{\lfloor(2l+1)n/(2(\gamma-1))\rfloor}(\cos(2\pi k/n)+\cos(2\pi\gamma k/n)).
\label{Ecos}
\end{align}
Writing the above relation in terms of Dirichlet kernels, it comes
\begin{align}
E(C^{1,\gamma}_n)&=2\sum_{m\in\{1,\gamma\}}\bigg(D_{\lfloor n/(2(\gamma+1))\rfloor}(2\pi m/n)\nonumber\\
&\ \ \ \qquad\qquad+\sum_{l=0}^{\lceil\gamma/2\rceil-2}\left(D_{\lfloor(2l+3)n/(2(\gamma+1))\rfloor}(2\pi m/n)-D_{\lfloor(2l+1)n/(2(\gamma-1))\rfloor}(2\pi m/n)\right)\nonumber\\
&\ \ \ \qquad\qquad-\sum_{l=0}^{\lceil\gamma/2\rceil-1}\left(D_{\lfloor(2l+1)n/(2(\gamma-1))\rfloor}(2\pi m/n)-D_{\lfloor(2l+1)n/(2(\gamma+1))\rfloor}(2\pi m/n)\right)\bigg).
\label{ED1}
\end{align}
Hence
\begin{align}
E(C^{1,\gamma}_n)=\sum_{m\in\{1,\gamma\}}\bigg(4&\sum_{l=0}^{\lceil\gamma/2\rceil-1}D_{\lfloor(2l+1)n/(2(\gamma+1))\rfloor}(2\pi m/n)\nonumber\\
&-4\sum_{l=0}^{\lceil\gamma/2\rceil-2}D_{\lfloor(2l+1)n/(2(\gamma-1))\rfloor}(2\pi m/n)-2D_{\lfloor n/2\rfloor}(2\pi m/n)\bigg).
\label{ED2}
\end{align}
The formula follows from the fact that $D_{\lfloor n/2\rfloor}(2\pi m/n)=0$ for $m=1,\gamma$.\\
Let $n$ be even. As for the case when $n$ is odd, we write the energy as follow
\begin{equation*}
E(C^{1,\gamma}_n)=4(1+\delta_{\gamma\textnormal{ odd}})+4\sum_{k=1}^{n/2-1}\lvert\cos(2\pi k/n)+\cos(2\pi\gamma k/n)\rvert
\end{equation*}
where $\delta_{\gamma\textnormal{ odd}}=1$ if $\gamma$ is odd and $0$ otherwise.\\
For even $\gamma$, relations (\ref{Ecos}), (\ref{ED1}) and (\ref{ED2}) hold. The theorem then follows from the fact that $D_{n/2}(2\pi/n)=-1$ and $D_{n/2}(2\pi\gamma/n)=1$. For odd $\gamma$, we have
\begin{align*}
E(C^{1,\gamma}_n)&=8+4\sum_{k=1}^{\lfloor n/(2(\gamma+1))\rfloor}(\cos(2\pi k/n)+\cos(2\pi\gamma k/n))\\
&\ \ \ +4\sum_{l=0}^{\lceil\gamma/2\rceil-2}\sum_{k=\lfloor(2l+1)n/(2(\gamma-1))\rfloor+1}^{\lfloor(2l+3)n/(2(\gamma+1))\rfloor}(\cos(2\pi k/n)+\cos(2\pi\gamma k/n))\\
&\ \ \ -4\sum_{l=0}^{\lceil\gamma/2\rceil-2}\sum_{k=\lfloor(2l+1)n/(2(\gamma+1))\rfloor+1}^{\lfloor(2l+1)n/(2(\gamma-1))\rfloor}(\cos(2\pi k/n)+\cos(2\pi\gamma k/n))\\
&\ \ \ -4\sum_{k=\lfloor(2\lceil\gamma/2\rceil-1)n/(2(\gamma+1))\rfloor+1}^{n/2-1}(\cos(2\pi k/n)+\cos(2\pi\gamma k/n)).
\end{align*}
Expressing it in terms of Dirichlet kernels, it comes
\begin{align*}
E(C^{1,\gamma}_n)&=4+2\sum_{m\in\{1,\gamma\}}\bigg(D_{\lfloor n/(2(\gamma+1))\rfloor}(2\pi m/n)\\
&\ \ \ \qquad\qquad+\sum_{l=0}^{\lceil\gamma/2\rceil-2}\left(D_{\lfloor(2l+3)n/(2(\gamma+1))\rfloor}(2\pi m/n)-D_{\lfloor(2l+1)n/(2(\gamma-1))\rfloor}(2\pi m/n)\right)\\
&\ \ \ \qquad\qquad-\sum_{l=0}^{\lceil\gamma/2\rceil-2}\left(D_{\lfloor(2l+1)n/(2(\gamma-1))\rfloor}(2\pi m/n)-D_{\lfloor(2l+1)n/(2(\gamma+1))\rfloor}(2\pi m/n)\right)\\
&\ \ \ \qquad\qquad-D_{n/2-1}(2\pi m/n)+D_{\lfloor(2\lceil\gamma/2\rceil-1)n/(2(\gamma+1))\rfloor}(2\pi m/n)\bigg).
\end{align*}
The theorem follows from the fact that $D_{n/2-1}(2\pi m/n)=1$ for $m=1,\gamma$.
\end{proof}
A graph is called hyperenergetic if his energy is greater than the one of the complete graph $K_n$. The eigenvalues of $K_n$ are given by $n-1$ and $-1$ with multiplicity $n-1$, so that his energy is given by $E(K_n)=2(n-1)$.\\
The figure on the left below shows how the energy of $C^{1,\gamma}_n$ grows with respect to $n$ for $\gamma=8$. We see that it is not hyperenergetic and that the energy grows more or less linearly with respect to $n$. The figure on the right shows the energy of $C^{1,\gamma}_n$ with fixed $n$ as $\gamma$ varies. We observe that the energy stays more or less constant independently of $\gamma$.
\begin{figure}[H]
\label{graph}
\centering
\subfigure[Energy of $C^{1,8}_n$ (crosses) and of $K_n$ (circles) with respect to $n$]{\includegraphics[width=6.15cm]{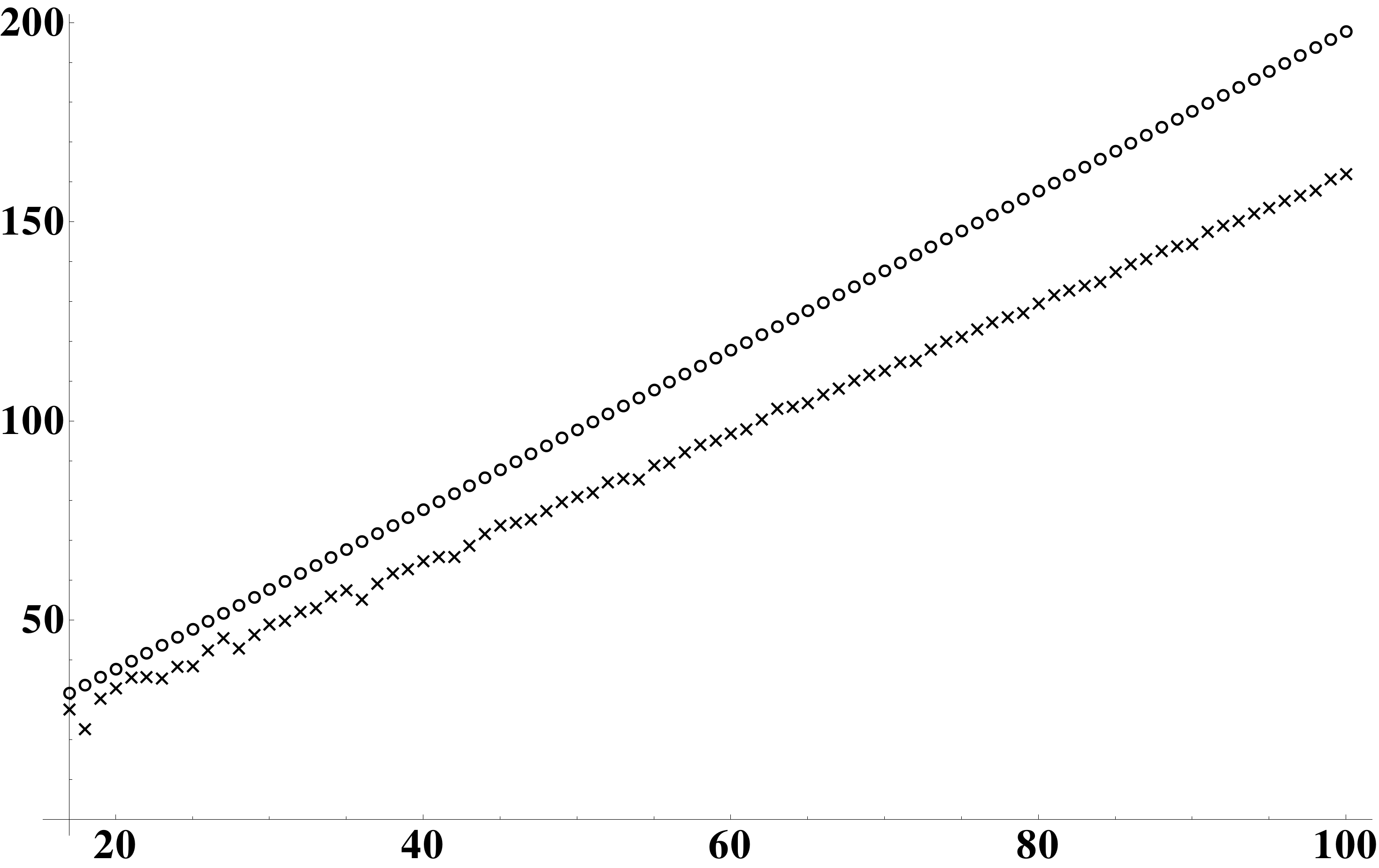}}
\hspace{2cm}
\subfigure[Energy of $C^{1,\gamma}_{150}$ with respect to $\gamma$]{\includegraphics[width=6.15cm]{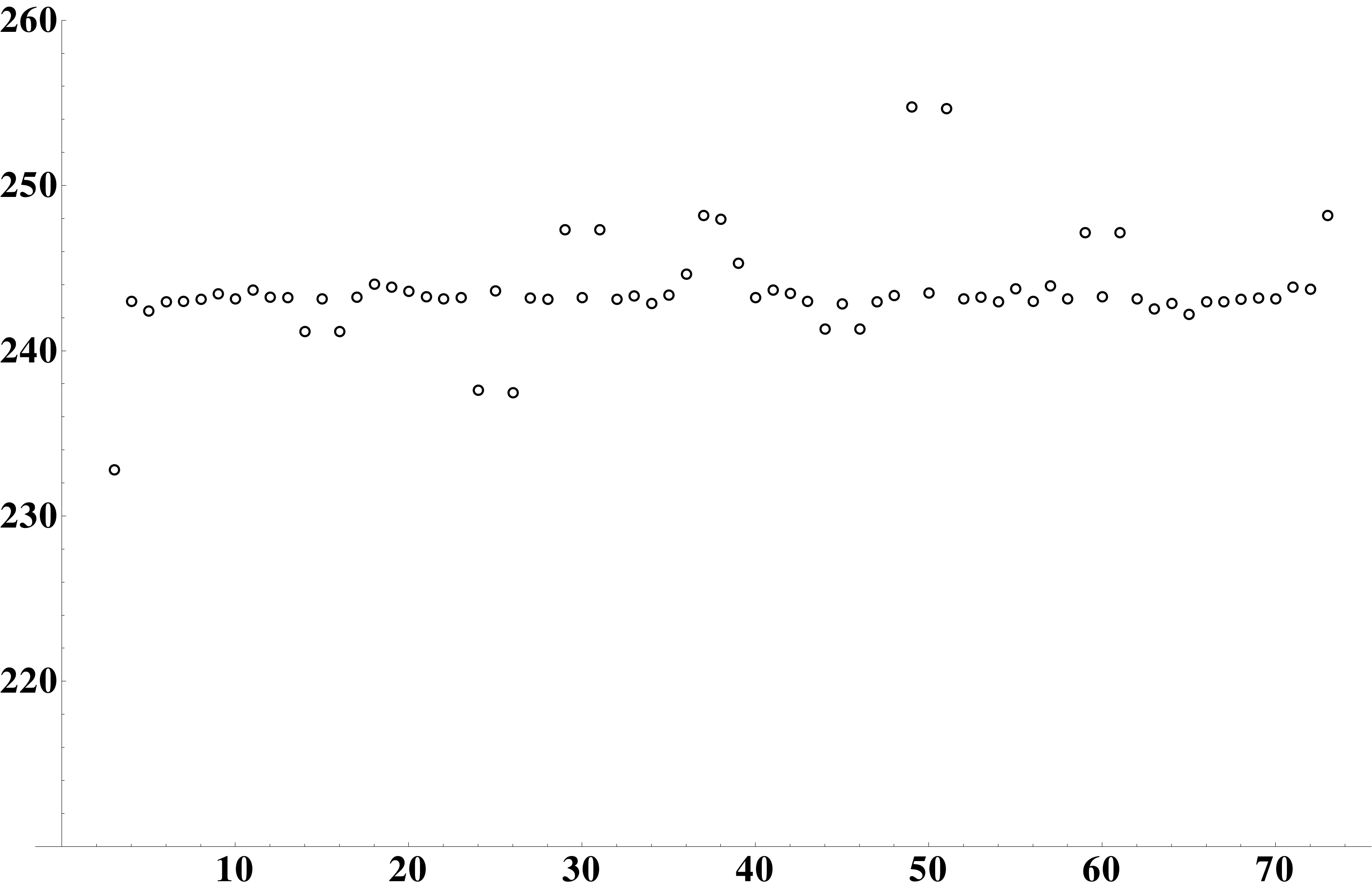}}
\caption{Energy of circulant graphs.}
\end{figure}
As a consequence of the theorem, we can carry out the sum of the Dirichlet kernels when the number of vertices is proportional to $2(\gamma-1)(\gamma+1)$.
\begin{corollary*}
Given integers $\gamma\geqslant3$ and $\alpha\geqslant1$, the energy of the circulant graph $C^{1,\gamma}_{2\alpha(\gamma-1)(\gamma+1)}$ is given by
\begin{align*}
E\Big(C^{1,\gamma}_{2\alpha(\gamma-1)(\gamma+1)}\Big)&=4\sum_{m\in\{1,\gamma\}}\bigg(\frac{\sin(\pi m(\lceil\gamma/2\rceil+1/(2\alpha(\gamma-1)))/(\gamma+1))\sin(\lceil\gamma/2\rceil\pi m/(\gamma+1))}{\sin(\pi m/(2\alpha(\gamma-1)(\gamma+1)))\sin(\pi m/(\gamma+1))}\\
&-\frac{\sin(\pi m(\lceil\gamma/2\rceil-1+1/(2\alpha(\gamma+1)))/(\gamma-1))\sin((\lceil\gamma/2\rceil-1)\pi m/(\gamma-1))}{\sin(\pi m/(2\alpha(\gamma-1)(\gamma+1)))\sin(\pi m/(\gamma-1))}\bigg).
\end{align*}
\end{corollary*}
\begin{proof}
Let $a\geqslant1$ and $K\geqslant0$ be integers. The sum over $k$ of Dirichlet kernels of index $(2k+1)a$ is given by
\begin{equation*}
\sum_{k=0}^KD_{(2k+1)a}(x)=\sum_{k=0}^K\frac{\sin((2k+1)a+1/2)x}{\sin(x/2)}.
\end{equation*}
By multiplying the summation by $\sin(ax)/\sin(ax)$ and using the trigonometric identity\linebreak $2\sin{\theta}\sin{\phi}=\cos(\theta-\phi)-\cos(\theta+\phi)$, it comes
\begin{equation*}
\sum_{k=0}^KD_{(2k+1)a}(x)=\frac{\cos(x/2)-\cos(((2K+2)a+1/2)x)}{2\sin(x/2)\sin(ax)}=\frac{\sin(((2K+2)a+1)x/2)\sin((K+1)ax)}{\sin(x/2)\sin(ax)}.
\end{equation*}
The corollary then follows by applying the above relation first with $a=\alpha(\gamma-1)$, $K=\lceil\gamma/2\rceil-1$, second with $a=\alpha(\gamma+1)$, $K=\lceil\gamma/2\rceil-2$, and $x=2\pi m/n$, $m\in\{1,\gamma\}$.
\end{proof}
In \cite{MR2069280}, the author considered the graphs $K_n-H$ where $K_n$ is the complete graph on $n$ vertices and $H$ is a Hamilton cycle of $K_n$ and asked whether these graphs are hyperenergetic. In \cite{stevanovic2005remarks}, the author showed that the energy of $K_n-H$ is given by
\begin{equation*}
E(K_n-H)=n-3+\sum_{k=1}^{n-1}\lvert1+2\cos(2\pi k/n)\rvert
\end{equation*}
and that as $n$ goes to infintiy, it is hyperenergetic. In the following proposition we give a formula for it for all $n\geqslant3$.
\begin{proposition*}
For all $n\geqslant3$, the energy of $K_n-H$ is given by
\begin{equation*}
E(K_n-H)=2(n-3-(\lfloor2n/3\rfloor-\lfloor n/3\rfloor))+2\frac{\sin((\lfloor n/3\rfloor+1/2)2\pi/n)-\sin((\lfloor2n/3\rfloor+1/2)2\pi/n)}{\sin(\pi/n)}.
\end{equation*}
\end{proposition*}
\begin{proof}
We have
\begin{align*}
\sum_{k=1}^{n-1}\lvert1+2\cos(2\pi k/n)\rvert&=\sum_{k=1}^{\lfloor n/3\rfloor}(1+2\cos(2\pi k/n))-\sum_{k=\lfloor n/3\rfloor+1}^{\lfloor2n/3\rfloor}(1+2\cos(2\pi k/n))\\
&\ \ \ +\sum_{k=\lfloor2n/3\rfloor+1}^{n-1}(1+2\cos(2\pi k/n))\\
&=n-3-2(\lfloor2n/3\rfloor-\lfloor n/3\rfloor)+2D_{\lfloor n/3\rfloor}(2\pi/n)-2D_{\lfloor2n/3\rfloor}(2\pi/n)\\
&\ \ \ +D_{n-1}(2\pi/n).
\end{align*}
Since $D_{n-1}(2\pi/n)=-1$, the proposition follows.
\end{proof}
By elementary analysis, one can show that $E(K_n-H)-2(n-1)$ is increasing in $n$. As a consequence, we find that $K_n-H$ are hyperenergetics for all $n\geqslant10$. This has been previously found in \cite{stevanovic2005remarks}.

\nocite{*}
\bibliographystyle{plain}
\bibliography{bibliography}

\end{document}